\documentclass{daj}

\usepackage{amssymb,amsmath,amsthm}

\newtheorem{proposition}[thm]{Proposition}
\newtheorem{theorem}[thm]{Theorem}

\newtheorem{conjecture}[thm]{Conjecture}

\dajAUTHORdetails{%
  title = {Fuglede's conjecture holds on cyclic groups ${\mathbb{Z}}_{pqr}$}, 
  author = {Ruxi Shi},
  plaintextauthor = {Ruxi Shi},
  plaintexttitle = {Fuglede's conjecture holds on cyclic groups Zpqr}, 
  copyrightauthor = {R. Shi},
  keywords = {Spectral set, tiling, Fuglede's conjecture, vanishing
  	sums of roots of unity},
}   

\dajEDITORdetails{%
   year={2019},
   number={14},
   received={24 July 2018},   
   revised={23 January 2019},    
   published={15 October 2019},  
   doi={10.19086/da.10570},       
}   

\def\ZN{\mathbb{Z}_{N}}
\def\Phip{\Phi_{p}(X^{N/p})}
\def\Phiq{\Phi_{q}(X^{N/q})}

\def\Z{\mathbb{Z}}
\def\ZZ{\mathcal{Z}}
\def\Zpqr{\mathbb{Z}_{pqr}}
\def\Ta{{\rm (T1)}}
\def\Tb{{\rm (T2)}}
\def\Zppp{\mathbb{Z}_{p_1p_2\dots p_k}}

\begin{document}

\begin{frontmatter}[classification=text]

\title{Fuglede's conjecture holds on cyclic groups ${\mathbb{Z}}_{pqr}$} 

\author[ruxi]{Ruxi Shi\thanks{Supported by the Centre of Excellence in Analysis and Dynamics Research funded by the
		Academy of Finland.}}

\begin{abstract}
Fuglede's spectral set conjecture states that a subset $\Omega$ of a locally compact abelian group $G$ tiles the group by translation if and only if there exists a subset of continuous group characters which is an orthogonal basis of $L^2(\Omega)$. We prove that Fuglede's conjecture holds on cyclic groups $\Zpqr$ with $p,q,r$ distinct primes.
\end{abstract}
\end{frontmatter}

\section{Introduction}
Let $G$ be a locally compact abelian group and let $\widehat{G}$ be its dual group consisting of all continuous group characters. Let $\Omega$ be a Borel measurable subset in $G$ with positive finite Haar measure.
We say that the set $\Omega$ is a {\em spectral set} if there exists a subset $\Lambda \subset \widehat{G}$ which is
an orthogonal basis of the Hilbert space $L^2(\Omega)$, and that $\Omega$ is  a
{\em  tile}  of $G$ by translation  if there exists a  set $T \subset G$ of translates
such that $\sum_{t\in T} 1_\Omega(x-t) =1$ for almost all $x\in G$, where $1_A$ denotes the indicator function of a set $A$. 

In the case where $G=\mathbb{R}^d$, Fuglede \cite{f}  formulated the following conjecture.

\begin{conjecture} A Borel set $\Omega\subset \mathbb{R}^d$ of positive and finite Lebesgue measure is a spectral set if and only if it is a tile. \end{conjecture}

\noindent Fuglede's conjecture has
attracted considerable attention over the last decades.  Many
positive results were obtained before Tao \cite{t} disproved the conjecture by showing that the
direction ``Spectral $\Rightarrow$ Tiling'' does not hold when $d\ge 5$.
Now it is known that the conjecture is false in both directions for $d \geq 3$ \cite{fmm,k,km,m}.  However, the conjecture is still open in lower dimensions ($d=1, 2$).

For any locally compact abelian group $G$, it is natural to formulate the following conjecture, also called Fuglede's conjecture (or spectral set conjecture) in $G$.

\begin{conjecture} A Borel set $\Omega\subset G$ of positive and finite Haar measure is a spectral set if and only if it is a tile. \end{conjecture}

\noindent In its full generality, this conjecture is far from being
proved and is false for some specific groups as Tao \cite{t} showed. The question becomes for which group $G$, Fuglede's conjecture holds. For an integer $N\ge 1$, the ring of integers modulo $N$ is denoted by $\ZN=\mathbb{Z}/N\mathbb{Z}$. We now know that Fuglede's conjecture holds on $\mathbb{Z}_{p^n}$\cite{l2, ffs}, $\mathbb{Z}_p \times \mathbb{Z}_p$ \cite{imp}, $p$-adic field $\mathbb{Q}_p$ \cite{ffs, ffls} and $\mathbb{Z}_{p^nq}$ with $n\ge 1$ \cite{mk}. 

Borrowing the notation from \cite{dl}, write $S$-$T(G)$, respectively $T$-$S(G)$, if  the direction ``Bounded spectral sets $\Rightarrow$Tiles", respectively ``Bounded tiles $\Rightarrow$ Spectral sets", holds in $G$.  The following relations are well established \cite{lw,pw,ls,m,fmm,km,dj}:
$$
T\text{-}S(\mathbb{R})\Longleftrightarrow T\text{-}S(\mathbb{Z}) \Longleftrightarrow T\text{-}S(\mathbb{Z}_N) ~\text{for all}~N,
$$
and 
$$
S\text{-}T(\mathbb{R})\Longrightarrow S\text{-}T(\mathbb{Z}) \Longrightarrow S\text{-}T(\mathbb{Z}_N) ~\text{for all}~N. 
$$
We refer to \cite{dl} for an overview of the literature.
From the above, we see that the cyclic groups play important roles in the study of the spectral set conjecture in $\mathbb{R}$. However, when we focus on this conjecture on cyclic groups, it is only known that this conjecture holds for $\mathbb{Z}_{p^n}$ for $p$ prime and very recently for $\mathbb{Z}_{p^nq}$ for $p,q$ distinct primes. As for the direction $T$-$S$ alone, we know that \L aba \cite{l2} proved $T$-$S(\mathbb{Z}_{p^nq^m})$ for $p,q$ distinct prime and that \L aba\footnote{https://terrytao.wordpress.com/2011/11/19/some-notes-on-the-coven-meyerowitz-conjecture/
	$\sharp$comment-121464} and Meyerowitz\footnote{https://terrytao.wordpress.com/2011/11/19/some-notes-on-the-coven-meyerowitz-conjecture/
	$\sharp$comment-112975} proved  in comments on Tao's blog \cite{t2} 	
	 that a tile in a cyclic group whose order is square-free always tiles by a subgroup.

In this paper, we prove that the spectral set conjecture holds in $\Z_{pqr}$ with $p,q,r$ distinct primes. This result relies heavily on the structure of vanishing sums of roots of unity, which was originally shown in \cite{ll} by Lam and Leung. 
Such structure is useful in the study of the zeros of the mask polynomials. Let $A$ be a multi-set in $\ZN$. Recall that the \textit{mask polynomial} of $A$ is defined to be
$$
A(X)=\sum_{a\in A} m_a X^a,
$$
where $m_a$ is the multiplicity of $a$ in $A$. It is clear that the degree of $A(X)$ is at most $N-1$. Sometimes, the mask polynomial $A(X)$ is regarded as the polynomial in $\Z[X]/(X^N-1)$.  Actually, any integer polynomial of degree at most $N-1$ with non-negative coefficients is a mask polynomial of some multi-set in $\ZN$. Observe that $A$ is a subset in $\ZN$ if and only if the coefficients of $A(X)$ are $0$ or $1$.

Let $\Phi_n$ be the cyclotomic polynomial of order $n$. Let $S$ be the set of prime powers dividing $N$. Define
$$
S_A=\{s\in S: \Phi_s(X)|A(X) \}.
$$
Following Coven and Meyerowitz \cite{cm}, we say that the set $A$ satisfies the condition $\Ta$ if
$$
\sharp A=\prod_{s\in S_A} \Phi_s(1), 
$$
and that it satisfies the condition $\Tb$ if $\Phi_{s_1s_2\dots s_m}(X)$ divides $A(X)$ whenever $s_1, s_2, \dots, s_m\in S_A$ are powers of distinct primes.



Now we state our main result.
\begin{theorem}\label{main thm}
	Let $A\subset \mathbb{Z}_{pqr}$ with $p,q,r$ distinct primes. Then the following are equivalent.
	\begin{itemize}
		\item [(1)] The set $A$ satisfies $\Ta$ and $\Tb$.
		\item [(2)] The set $A$ is a spectral set.
		\item [(3)] The set $A$ tiles $\mathbb{Z}_{pqr}$ by translation.
	\end{itemize}
\end{theorem}

In fact, Coven and Meyerowitz \cite{cm} proved that the conditions $\Ta+\Tb$ imply tiles, i.e. $(1) \Rightarrow (3)$, and that tiles satisfy the condition $\Ta$; \L aba \cite{l2} proved that conditions $\Ta+\Tb$ imply spectral sets, i.e. $(1) \Rightarrow(2)$; the implication $(3)\Rightarrow (1)$ follows from \L aba and Meyerowitz's comments on Tao's blog \cite{t2}. Our novel contribution here is to prove that spectral sets satisfy the condition $\Ta$ and $\Tb$, i.e. $(2)\Rightarrow(1)$. The method we used here is to analyse the structure of $p$-cycles, $q$-cycles and $r$-cycles by identifying the cyclic group $\Z_{pqr}$ with the product space $\Z_p\times \Z_q\times \Z_r$.

We organize the paper as follows. In Section \ref{Section: Prelimimaires}, we revisit the definitions of tiles and spectral sets in cyclic groups $\ZN$ and introduce prime-cycles in multi-sets. In Section \ref{Section:Prime-cycles}, we give an equivalent description of prime-cycles by identifying  $\Z_{p_1 p_2\cdots p_k}$ with $\mathbb{Z}_{p_1} \times \Z_{p_2} \times \cdots \Z_{p_k}$. In Section \ref{Section:Tiles=T1+T2}, we prove that a tile in $\Zppp$ satisfies $\Ta$ and $\Tb$. In Section \ref{Section:Spectral=T1+T2}, we prove that a spectral set in $\Zpqr$ satisfies $\Ta$ and $\Tb$.

\section{Preliminaries}\label{Section: Prelimimaires}
In this section, we first give some equivalent definitions of spectral sets and tiles in $\ZN$. We then study the zeros of mask polynomials. At the end of this section, we study the structure of vanishing sums of roots of unity. Throughout this paper we denote by $\omega_N=e^{2\pi i /N}$, for $N\ge 1$, which is a primitive $N$-th root of unity.
\subsection{Spectral sets in $\ZN$}\label{Spectral set and tile}
We recall that a subset $A\subset \ZN$ is said to be \textit{spectral} if there is a set $B\subset \ZN$ such that 
$$
\{\omega_N^{bx}: b\in B \}
$$
forms an orthogonal basis for $L^2(A)$. In such a case, the set $B$ is called a \textit{spectrum} of $A$ and the pair $(A,B)$ is called a \textit{spectral pair}. Since the dimension of $L^2(A)$ is $\sharp A$, the pair $(A ,B)$ being a spectral pair is equivalent to that $\sharp A=\sharp B$ and $\{\omega_N^{bx}: b\in B \}$ is orthogonal on $L^2(A)$, that is to say, 
\begin{equation}\label{spectral condition}
\sharp A=\sharp B ~\text{and}~\sum_{a\in A} \omega_N^{(b-b')a}=0, ~\forall b,b'\in B, b\not=b'.
\end{equation}
This also means that the complex matrix $M=\left(\omega_N^{ba}\right)_{b\in B, a\in A}$ is a complex Hadamard matrix, i.e. $M\overline{M}^{T}=(\sharp A)I$ where $A^T$ is the transpose of $A$ and $I$ is the identity matrix. Since $\sharp A=\sharp B$, it follows that $\overline{M}^TM=(\sharp B)I$. This implies that $B$ is also a spectral set and $A$ is its spectrum.

Obviously, the mask polynomial of $A$ is indeed the Fourier transform of the indicator function of $A$ as follows
\begin{equation}\label{the Fourier transform and Mask polynomial}
\widehat{1}_A(n)=\sum_{a\in A}\omega_N^{-an} =A(\omega_N^{-n}), ~\forall~n\in \ZN.
\end{equation}
Denote the zeros of the mask polynomial $A(X)$ by $$\mathcal{Z}_A=\{n\in \mathbb{Z}_N: A(\omega_N^n)=0 \}.$$ We restate the above equivalent definition of spectral sets as follows. 
\begin{proposition}\label{Equivalence of spectral set}
	Let $A, B\subset \ZN$. The following are equivalent.
	\begin{itemize}
		\item [(1)] The pair $(A,B)$ is a spectral pair.
		\item [(2)] The pair $(B,A)$ is a spectral pair.
		\item [(3)] $\sharp A=\sharp B$; $(B-B)\setminus \{0\}\subset \mathcal{Z}_A$.
	\end{itemize}
\end{proposition}

\subsection{Tiles in $\ZN$}
Recall that a subset $A\subset \ZN$ is said to be a \textit{tile} if there is a set $T\subset \ZN$ such that 
$$
\bigsqcup_{t\in T} (A+t)=\ZN.
$$
In such case, the set $T$ is called a \textit{tiling complement} of $A$ and the pair $(A,T)$ is called a \textit{tiling pair}. Using the language of mask polynomials, we have the following equivalent definitions of tiles in $\ZN$.

\begin{lem}[Lemma 1.3, \cite{cm}]\label{lem:tiling equivalent}
	Let $N$ be a positive integer. Let $A,B$ be multi-sets in $\ZN$. Then the following statements are equivalent. Each forces $A$ and $B$ to be sets such that $(\sharp A)(\sharp B)=N$.
	\begin{itemize}
		\item [(1)] $(A,B)$ is a tiling pair.
		\item [(2)] $A\oplus B=\ZN$.
		\item [(3)] $A(X)B(X)=1+X+X^2+\cdots+X^{N-1}\mod X^N-1.$
	\end{itemize}
\end{lem}

As we have mentioned before, Coven and Meyerowitz proved that a set in $\ZN$ satisfying the conditions $\Ta$ and $\Tb$ is a tile, which shows $(1)\Rightarrow (3)$ in Theorem \ref{main thm}. 

\begin{theorem}[Theorem A, \cite{cm}]\label{thm:Coven and Meyerowitz}
	Let $A\subset \ZN$. If $A$ satisfies $(T1)$ and $(T2)$, then $A$ is a tile of $\ZN$.
\end{theorem}

\subsection{Zeros of mask polynomial}
Let $N\ge 1$. Let $\ZN^\star$ denote the group of all invertible elements in $\ZN$. It is classical that the Galois group
$
\text{Gal}(\mathbb{Q}(\omega_N)/\mathbb{Q})
$ is isomorphic to $\ZN^\star$ and the isomorphism is induced by $\sigma(\omega_N)=\omega_N^g$ for $g\in \ZN^\star$.
It follows that we have a natural action of $\text{Gal}(\mathbb{Q}(\omega_N)/\mathbb{Q})$ on the values of the mask polynomial of $A$, defined by
$$
\sigma(A(X))=A(X^g),
$$
for $\sigma\in \text{Gal}(\mathbb{Q}(\omega_N)/\mathbb{Q})$ which is determined by $\sigma(\omega_N)=\omega_N^g$ for $g\in \Z_N^\star$. 

The following lemma shows that we only need to study the zeros of $A(\omega_N^d)=0$ by restricting $d$ as a divisor of $N$. 
\begin{lem}\label{Lem: zeros N/p^d of mask polynomial}
	Let $A$ be a subset in $\ZN$. Let $p$ be a prime factor of $N$. Let $a\in \ZN$. Then we have the following.
	\begin{itemize}
		\item [(1)] $A(\omega_N^a)=0$ if and only if $A(\omega_N^{ag})=0$ for any $g\in \ZN^\star$.
		\item [(2)] For any $d\mid N$, we have the equivalence 
		$$
		A(\omega_N^d)=0 \Longleftrightarrow \Phi_{N/d}(X)|A(X) \mod X^N-1.
		$$
		\item [(3)] Suppose that $\sharp \{p^d\in \ZN: A(\omega_N^{N/p^d})=0\}=k$. Then $p^k$ divides $\sharp A$.
	\end{itemize}
	
\end{lem}
\begin{proof}
	For $g\in \Z_N^\star$, let $\sigma\in \text{Gal}(\mathbb{Q}(\omega_N)/\mathbb{Q})$ which is determined by $g$. It follows that 
	$
	\sigma(A(\omega_N^a))=A(\omega_N^{ag}).
	$
	Since $\sigma(0)=0$, we deduce $(1)$. Since the cyclotomic polynomial of order $N/d$ is the monic minimal polynomial of $\omega_N^d$, we obtain $(2)$.
	By definition, $\prod_{s\in S_A} \Phi_s(X)$ divides $A(X)$ in $\mathbb{Z}[X]$. Putting $X=1$, we get $(3)$.
\end{proof} 

The above lemma shows that $\ZZ_A$ is invariant by multiplying any element in $\ZN^\star$.

\subsection{Vanishing sums of roots of unity}
Let $d$ be a factor of $N$. Following \cite{mk}, a \textit{$d$-cycle} is a coset of the cyclic subgroup of order $N/d$ in $\ZN$, that is, a set of the form
$$
\{j, j+N/d, j+2N/d, \dots, j+(d-1)N/d \},
$$ 
for some $j\in \ZN$.
Moreover, we say that such $d$-cycle is a \textit{prime-cycle} if $d$ is a prime. For interpretation of cycles, one may refer to \cite{st} where Steinberger used the tensor product representation (given here in Lemma \ref{lem:goup isomorphism}) to provide a useful geometric interpretation of cycles.

Let $A$ be a multi-set in $\ZN$. Let $n\ge 1$. Denote by $n\cdot A$ (or sometimes $nA$) the multi-set consisting of elements $na\in \ZN$ counting the multiplicity  for all $a\in A$. We are concerned whether the multi-set $n\cdot A$ is a union of prime-cycles. 
If we assume that $N/n$ has at most two prime divisors, say $p$ and $q$, the following lemma tells us that $n\cdot A$ must be a union of $p$-cycles and $q$-cycles whenever $A(\omega_N^n)=0$.

\begin{proposition}[Lemma 2.5, Proposition 2.6, \cite{mk}] \label{VanishingTwoPolynomial}
	Let $n$ be a factor of $N$ such that $N/n$ has at most two prime divisors, say $p$ and $q$. If $A(\omega_N^n)=0$, then
	$$
	A(X^n)\equiv P(X^n)\Phip+Q(X^n)\Phiq \mod X^N-1,
	$$
	where $P$ and $Q$ have nonnegative coefficients. Moreover, if $N/n$ has only one prime divisor, say $p$, then $Q\equiv 0$; if $A(\omega_N^{nq^b})\not=0$ for some $b>0$, then $Q\not\equiv 0$.
\end{proposition}

Obviously, the converse of the previous proposition also holds, even for $N/n$ having more than two prime divisors. However, the previous proposition is not true when $N/n$ has at least three prime divisors.
For example, considering the multi-set $A$ defined by the following mask polynomial in $\ZN$ where $N=pqr$ with $p,q,r$ distinct primes,
\begin{align*}
A(X)=&(X^{qr}+X^{2qr}+\dots+X^{(p-1)qr})(X^{pr}
+X^{2pr}+\dots+X^{(q-1)pr})\\
&+(X^{pq}+X^{2pq}+\dots+X^{(r-1)pq}),
\end{align*} 
we have $A(\omega_N)=0$ but $A$ cannot be expressed as a union of $p$-, $q$- and $r$-cycles \cite{ll}.

Nevertheless, for general $N>1$, Lam and Leung \cite{ll} proved that if $A(\omega_N)=0$, then $\sharp A$ is a nonnegative integer linear combination of the prime divisors of $N$.
\begin{theorem}[Main theorem, \cite{ll}]\label{Lem:Lam and Leung}
	Let $N$ be a positive integer. If $A(\omega_N)=0$, then there exist integers $n_p\ge 0$ for all $p$ prime with $p~|~N$ such that $\sharp A=\sum_{p|N} n_p p$.
\end{theorem}

\section{Prime-cycles in $\Z_N$ for $N$ square free}\label{Section:Prime-cycles}
In this section, we identify the cyclic group $\Z_N$ for $N=p_1p_2\cdots p_k$ square free,  with the product space $\mathbb{Z}_{p_1} \times \Z_{p_2} \times \cdots \Z_{p_k}$ and show the equivalent description of prime cycles in $\Z_N$ by using their representations in  $\mathbb{Z}_{p_1} \times \Z_{p_2} \times \cdots \Z_{p_k}$.

The following lemma is classical and of independent interest. We provide the proof in order to make our paper self-contained. 

\begin{lem}\label{lem:goup isomorphism}
	Let $N=p_1p_2\cdots p_k$. Then there is a group isomorphism 
	\begin{align*}
	\phi: (\ZN,+) & \longrightarrow  (\mathbb{Z}_{p_1} \times \Z_{p_2} \times \cdots \Z_{p_k},+),\\
	x &\longmapsto ~~~ (x_1, x_2, \dots, x_k),	
	\end{align*}
	where $x_j\equiv x \mod p_j$ for all $1\le j\le k$.
\end{lem} 
\begin{proof}
	Let $x,y\in \ZN$. Denote by $\phi(x)=(x_1, x_2, \dots, x_k)$ and $\phi(y)=(y_1, y_2, \dots, y_k)$. We observe that 
	$$
	x+y \equiv x_j+y_j \mod  p_j, ~\forall~1\le j\le k,
	$$
	which implies that $\phi(x+y)=\phi(x)+\phi(y)$. Thus $\phi$ is a group homomorphism. On the other hand, if $\phi(x)=(0,0,\dots,0)$ which means that 
	$$
	x\equiv 0 \mod  p_j, ~\forall~1\le j\le k,
	$$
	then $x=0$. This implies that $\phi$ is injective. Due to the Chinese remainder theorem, it is also surjective. Therefore, we conclude that $\phi$ is an isomorphism.  
\end{proof}

For an element $x\in \ZN$, we sometimes write $(x_1,x_2,\dots, x_k)$, which is actually $\phi(x)$, to represent $x$. In what follows, we identify $\ZN$ with the set $\{0,1,\dots, N-1 \}$ whenever we do not concentrate on the addition on $\ZN$. Now we study the equivalent definition of prime-cycles in $\Z_N$.

\begin{lem}\label{lem:prime cycle 1}
	Let $L$ be a multi-set in $\ZN$. Then $L$ is a $p_1$-cycle if and only if it has the form
	\begin{equation}\label{eq:prime cycles form}
	\{(\ell, x_2, \dots, x_k ): 0\le \ell\le p_1-1 \},
	\end{equation}
	for some $x_j\in \Z_{p_j}$ for $2\le j\le k$. Moreover, if $L$ is a $p_1$-cycle, then for any $y\in \ZN$ with $p_1\nmid~y$, $yL$ is also a $p_1$-cycle.
\end{lem}
\begin{proof}
	By definition, $L$ is a $p_1$-cycle if and only if $L$ has the form $\{x, x+N/p_1, x+2N/p_1, \dots, x+(p_1-1)N/p_1 \}$ for some $x\in \ZN$. 
	We observe that 
	$$
	\ell N/p_1\equiv 0 \mod p_j,~\forall~2\le j\le k, 0\le \ell\le p_1,
	$$
	and that
	$
	\{\ell N/p_1: 0\le \ell \le p_1 \}
	$ forms a complete set of residues modulo $p_1$.
	Thus for $x=(x_1,x_2, \dots, x_k)\in \ZN$, we deduce that the set $\{x, x+N/p_1, x+2N/p_1, \dots, x+(p_1-1)N/p_1 \}$ is exactly of the form $(\ref{eq:prime cycles form})$.
	
	On the other hand, it is not hard to check that if $L$ has the form $(\ref{eq:prime cycles form})$, then $yL$ also has the form $(\ref{eq:prime cycles form})$ for any  $y\in \ZN$ with $p_1 \nmid ~y$. This completes the proof.
\end{proof}

Now we prove a criterion for a multi-set in $\Z_N$ having prime-cycles. 
\begin{lem}\label{lem:prime cycle 2}
	Let $A$ be a multi-set in $\ZN$. Let $2\le m\le k$. Then $p_mp_{m+1}\cdots p_k A$ has a $p_1$-cycle if and only if the multi-set $A$ contains a subset
	\begin{equation*}
	\{(\ell, x_2, \dots, x_{m-1}, x_m^{(\ell)} \dots,  x_k^{(\ell)}   ):  0\le \ell\le p_1-1  \}.
	\end{equation*}
\end{lem}
\begin{proof}
	Observe that $p_mp_{m+1}\cdots p_k \ZN$ is the set
	$$
	\left\{(x_1, x_2, \dots, x_{m-1}, 0, 0, \dots,  0 ): x_j\in \Z_{p_j}~\text{for}~1\le j\le m-1  \right\}.
	$$
	By Lemma \ref{lem:prime cycle 1}, we complete the proof.
\end{proof}
The above lemma is useful to analyse the structure of multi-sets when it has some zeros. Moreover, it is the crucial technique in the proof of Theorem \ref{main thm}.


\section{Tiles in $\Zppp$ $\Rightarrow$ $\Ta+\Tb$}\label{Section:Tiles=T1+T2}
In this section, we show that tiles in  $\Zppp$ satisfy the conditions $\Ta$ and $\Tb$. After we obtained this result, we are informed by Romanos-Diogenes Malikiosis that the result is not new and is a consequence of the argument \cite{t2}: a tile in a cyclic group whose order is square-free always tiles by a subgroup. We give the proof here for completeness. Actually, we prove the following proposition.

\begin{proposition}\label{main thm 2}
	Let $A$ be a subset in $\Zppp$ with $p_1, p_2, \dots, p_k$ distinct primes. Then the following are equivalent.
	\begin{itemize}
		\item [(i)] The set $A$ is a tile of $\Zppp$ by translation.
		\item [(ii)] The set $A$ satisfies the conditions $\Ta$ and $\Tb$.
		\item [(iii)] The set $A$ has the form 
		$$
		\{(\vec{n},\vec{y}_{\vec{n}}): \vec{n}\in \Z_{p_1'p_2'\cdots p_{\ell}'}  \}.
		$$
		where $1\le \ell\le k$, $(p_1',p_2', \dots, p_k' )$ is a permutation of $(p_1,p_2, \dots, p_k )$ and $\vec{y}_{\vec{n}}\in \Z_{p_{\ell+1}'p_{\ell+2}'\cdots p_{k}'}$ for all $\vec{n}\in \Z_{p_1'p_2'\cdots p_{\ell}'}$. 
	\end{itemize}
\end{proposition}





We will need the following lemma which is due to Tijdeman \cite{tij} and is also included in \cite{cm}. We review the proof for completeness.
\begin{lem}\label{Lem:nA is also a tile}
	Let $N$ be a positive integer. Let $(A,B)$ be a tiling pair in $\ZN$. Then for any positive integer $n$ with $(n, \sharp A)=1$, $nA$ is a subset in $\ZN$ and $(nA, B)$ is also a tiling pair in $\ZN$.
\end{lem}
\begin{proof}
	Since $(A,B)$ is a tiling pair in $\ZN$, by Lemma \ref{lem:tiling equivalent}, we have
	\begin{equation}
	A(X)B(X)=1+X+X^2+\cdots+X^{N-1}\mod X^N-1.
	\end{equation}
	We write $n=p_1^{\alpha_1}p_2^{\alpha_2}\cdots p_k^{\alpha_k}$ where $\alpha_1, \alpha_2, \dots, \alpha_k\ge 1$ and $p_1, p_2, \dots, p_k$ are distinct primes. By Lemma 3.1 \cite{cm}, we have
	$$
	A(X^{p_1})B(X)=1+X+X^2+\cdots+X^{N-1}\mod X^N-1.
	$$
	Observing that the mask polynomial of $p_1A$ is $A(X^{p_1}) ~(\text{mod}~ X^N-1)$, by Lemma \ref{lem:tiling equivalent}, we obtain that $(p_1A, B)$ is a tiling pair. It follows that $\sharp (p_1A)=\sharp A$. Since $p_j$ do not divide $\sharp A$ for all $1\le j\le k$, repeating the same reason, we get that $(nA, B)$ is a tiling pair and consequently $nA$ is a subset in $\ZN$.
\end{proof}




\begin{proof}[Proof of Proposition \ref{main thm 2}] Denote by $N=p_1p_2\cdots p_k$. Since it is proved that $(ii)\Rightarrow (i)$ by Theorem \ref{thm:Coven and Meyerowitz}, it remains to prove here that $(i)\Rightarrow (iii)$ and $(iii)\Rightarrow (ii)$.
	\vspace{7pt}
	
	$(i)\Rightarrow (iii):$
	Suppose that $A$ is a tile in $\ZN$. Then there is $1\le \ell\le k$ and a permutation $(p_1',p_2', \dots, p_k' )$ of $(p_1,p_2, \dots, p_k )$ such that
	$
	\sharp A=p_1'p_2'\cdots p_{\ell}'.
	$
	By Lemma \ref{Lem:nA is also a tile}, $p_{\ell+1}'p_{\ell+2}'\cdots p_{k}'A$ is a subset in $\ZN$.  Observe that the set $p_{\ell+1}'p_{\ell+2}'\cdots p_{k}'A$ has the form
	$$
	\left\{(x,\vec{0}_{\Z_{p_{\ell+1}'p_{\ell+2}'\cdots p_{k}'}}): x\in S\right\},
	~\text{for some} ~S \subset \Z_{p_1'p_2'\cdots p_{\ell}'},
	$$
	where $\vec{0}_{\Z_{p_{\ell+1}'p_{\ell+2}'\cdots p_{k}'}}$ stands for $(0,0,\dots, 0)\in \Z_{p_{\ell+1}'}\times \Z_{p_{\ell+2}'}\times\cdots \Z_{p_k'}$. Since $\sharp A=p_1'p_2'\cdots p_{\ell}'$, it follows that $S$ has to be $\Z_{p_1'p_2'\cdots p_{\ell}'}$. Thus $A$ has the form
	$$
	A=\{(\vec{n},\vec{y}_{\vec{n}}): \vec{n}\in \Z_{p_1'p_2'\cdots p_{\ell}'}  \},
	$$
	where $\vec{y}_{\vec{n}}\in \Z_{p_{\ell+1}'p_{\ell+2}'\cdots p_{k}'}$ for all $\vec{n}\in \Z_{p_1'p_2'\cdots p_{\ell}'}$.
	
	\vspace{10pt}

	$(iii)\Rightarrow (ii):$ Let $N'=p_1'p_2'\dots p_{\ell}'$. It is not hard to see that $\sharp A=N'$ and for any $1\le j\le \ell$,
	$$
	(N/p_j')\cdot A=(N'/p_j') \left\{ \left(\vec{0}_{\Z_{p_1'p_2'\cdots p_{j-1}'}}, i, \vec{0}_{\Z_{p_{j+1}'p_{\ell+2}'\cdots p_{k}'}} \right): 0\le i\le p_j'-1  \right\}.
	$$
	It follows that $(N/p_j')\cdot A$ is a union of $p_j'$-cycles. By Proposition \ref{VanishingTwoPolynomial}, $N/p_j'\in \ZZ_A$ for all $1\le j\le \ell$. It follows that $A$ satisfies $\Ta$. 
	It remains to prove that $A$ satisfies the condition $\Tb$. In fact, it is sufficient to prove that $p_1'p_2'\cdots p_j' p_{\ell+1}'p_{\ell+2}' \cdots p_{k}'\in \ZZ_A $ for any $0\le j\le \ell-1$. Actually, we observe that $p_1'p_2'\cdots p_j' p_{\ell+1}'p_{\ell+2}' \cdots p_{k}' A$ has the form
	$$
	p_1'p_2'\cdots p_j'\{ (\vec{0}_{\Z_{p_1'p_2'\cdots p_j'}},\vec{n}, \vec{0}_{\Z_{p_{\ell+1}'p_{\ell+2}'\cdots p_{k}'}} ):\vec{n}\in \Z_{p_{j+1}'p_{j+2}'\cdots p_{\ell}'}  \}.
	$$
	It is easy to see that $p_1'p_2'\cdots p_j' p_{\ell+1}'p_{\ell+2}' \cdots p_{k}' A$ is a union of $p_{\ell}'$-cycles. It follows that 
	$$p_1'p_2'\cdots p_j' p_{\ell+1}'p_{\ell+2}' \cdots p_{k}'\in \ZZ_A.$$
	 We conclude that $A$ satisfies the condition $\Tb$.

\end{proof}

\section{Spectral sets in $\Zpqr$ $\Rightarrow$ $\Ta+\Tb$}\label{Section:Spectral=T1+T2}
In this section, we prove that spectral sets in $\Zpqr$ satisfies the conditions $\Ta+\Tb$. This completes the proof of Theorem \ref{main thm}.

We first prove several technical lemmas. These lemmas which may also be of independent interest will be useful in the proof of Theorem \ref{main thm}. 
\begin{lem}\label{lem:1}
	Let $(A,B)$ be a spectral pair in $\Zpqr$. 
	If $pq\notin \ZZ_B$, then there exist a subset $S\subset \Z_p\times \Z_q $ and a function $f: \Z_p\times \Z_q \to \Z_r$ such that 
	\begin{equation}\label{eq:f(pq)}
	A=\{(x,y,f(x,y)): (x,y)\in S \}.
	\end{equation}
	Moreover, we have $\sharp A\le pq$, and equality holds if and only if $S=\Z_p\times \Z_q$.
\end{lem}
\begin{proof}
	If the set $A$ has two elements $(x,y,z)$ and $(x,y,z')$ with $z\not=z'$, then we have
	$$
	(x,y,z)-(x,y,z')=(0,0, z-z') \in pq \Z_{pqr}^\star.
	$$
	By Proposition \ref{Equivalence of spectral set}, we have $pq\in \ZZ_B$, which is a contradiction. It follows that for any $(x,y,z)\in A$, the value of $z$ is decided by the values $x$ and $y$ together. Thus $A$ has the form $(\ref{eq:f(pq)})$. Moreover, it is easy to see that $\sharp A=\sharp S$. Since $S$ is a subset of $\Z_p\times \Z_q$, we conclude that $\sharp A=pq$ if and only if $S=\Z_p\times \Z_q$. 
\end{proof}

The general case of the following lemma has been already proved by the implication $(iii)\Rightarrow (ii)$ in Proposition \ref{main thm 2}. Here, we use mask polynomials to give a different proof of this special case. We remark that the general case can also be proved by this method.

\begin{lem}\label{lem:3}
	Suppose that a set $A$ has the form (\ref{eq:f(pq)}) with $S=\Z_p\times \Z_q$. Then $A$ satisfies the conditions $\Ta$ and $\Tb$. 
\end{lem}
\begin{proof}
	We observe that $prA$ is a multi-set in $\Zpqr$ and its mask polynomial is
	\begin{equation}
	(prA)(X)=A(X^{pr})=\sum_{j\in \Z_q} pX^{jpr}.
	\end{equation}
	It follows that $(prA)(\omega_{pqr})=0$. Thus $pr\in \ZZ_A$. Similarly, we have $qr\in \ZZ_A$. Since $r\nmid \sharp A=pq$, we have $pq\notin \ZZ_A$. Therefore, we obtain that $A$ satisfies the condition $\Ta$.  On the other hand, it is easy to see that 
	\begin{equation}
	rA=\{ (x,y,0): (x,y)\in  \Z_p \times \Z_q \}.
	\end{equation}
	It follows that $rA$ is a union of $p$-cycles. Thus we obtain that $r\in \ZZ_A$ and conclude that $A$ satisfies the condition $\Tb$. 
	
\end{proof}

The following result is crucial for our proof of Theorem \ref{main thm}.
\begin{lem}\label{lem:lalala}
	Let $A$ be a multi-set in $\Zpqr$. If $(A-A)\cap pq\Zpqr^\star\not=\emptyset$ and there exists a $r$-cycle in $pA$, then $p>r$ and $(A-A)\cap q\Zpqr^\star\not=\emptyset$.
\end{lem}
\begin{proof}
	Since $pA$ has a $r$-cycle,	by Lemma \ref{lem:prime cycle 2}, we obtain that the multi-set $A$ has a subset $L$ which has the form
	$$
	\{(f(j), y, j): 0\le j\le r-1 \}
	$$
	for some $y\in \Z_q$ and some function $f: \Z_r \to \Z_p$. Since $(A-A)\cap pq\Zpqr^\star\not=\emptyset$, we obtain that for $(x,y,z), (x',y,z')\in A$, if $z\not= z'$, then $x\not=x'$. It follows that the function $f$ is injective and that $(A-A)\cap q\Zpqr^\star\not=\emptyset$. By the fact that $p,r$ are distinct primes,  we have $p>r$.
\end{proof}

By Proposition \ref{VanishingTwoPolynomial}, if $p\in \ZZ_A$ and $pr\notin \ZZ_A$, then $pA$ must have a $r$-cycle. Thus, the following lemma is a direct consequence of Lemma \ref{lem:lalala}.
\begin{lem}\label{lem:OLALALA}
	Let $(A,B)$ be a spectral pair in $\Zpqr$. Suppose that $p\in \ZZ_A$, $pr\notin \ZZ_A$ and $pq\notin \ZZ_B$. Then $p>r$ and $q\in \ZZ_B$.
\end{lem}

Now we begin to prove Theorem \ref{main thm} $(3)\Rightarrow (1)$. If $\{pq, pr, qr\} \subset \ZZ_A$, then by Lemma \ref{Lem: zeros N/p^d of mask polynomial}, the set $A$ has to be $\Zpqr$, which satisfies the conditions $\Ta$ and $\Tb$. It remains to consider the case where $\sharp (\ZZ_A \cap \{pq, pr, qr\})\le 2$. We then decompose the proof of Theorem \ref{main thm} $(3)\Rightarrow (1)$ into three situation: $\sharp (\ZZ_A \cap \{pq, pr, qr\}) =2$; $\ZZ_A \cap \{pq, pr, qr\} =\emptyset$; $\sharp (\ZZ_A \cap \{pq, pr, qr\}) =1$. We will prove that $A$ satisfies the conditions $\Ta$ and $\Tb$ in each situation.

\subsection{Case 1: $\sharp (\ZZ_A \cap \{pq, pr, qr\}) =2$}\label{subsec: 2}
Without loss of generality, we suppose that $qr, pr\in \ZZ_A$ and $pq\notin \ZZ_A$. It follows that $pq~|~\sharp A=\sharp B$. If $pq\in \ZZ_B$, then by Lemma \ref{Lem: zeros N/p^d of mask polynomial}, we have that $r$ divides $\sharp B$ and consequently $\sharp A=\sharp B=pqr$. This implies that $A=\Zpqr$ and as a result $\{pq, pr, qr\} \subset \ZZ_A$ , which is impossible. Thus $pq\notin \ZZ_B$. By Lemma \ref{lem:1} and the fact that $pq$ divides $\sharp A$, we have $\sharp A=pq$ and
there exists a function $f: \Z_p\times \Z_q \to \Z_r$ such that 
\begin{equation}\label{eq:pq,pr 2}
A=\{(x,y,f(x,y)): (x,y)\in \Z_p \times \Z_q \}.
\end{equation}
By Lemma \ref{lem:3}, we conclude that $A$ satisfies the conditions  $\Ta$ and $\Tb$. 

\begin{rem}\label{Rem:2 is 2}
	In this case, by the fact that $pq\notin \ZZ_A$ and $\sharp B= \sharp A=pq$, we deduce by Lemma \ref{lem:1} that there exists a function $g: \Z_p\times \Z_q \to \Z_r$ such that 
	\begin{equation}\label{eq:pq,pr 3}
	B=\{(x,y,g(x,y)): (x,y)\in \Z_p \times \Z_q \}.
	\end{equation} 
	Hence, by Lemma \ref{lem:3}, we obtain that $\sharp (\ZZ_B \cap \{pq, pr, qr\}) =2$ and the spectrum $B$ also satisfies the conditions  $\Ta$ and $\Tb$. 
\end{rem}

\subsection{Case 2: $\ZZ_A \cap \{pq, pr, qr\} =\emptyset$}\label{subsec:0}
It follows from Lemma \ref{lem:1} that $\sharp B=\sharp A\le \min\{pq, pr, qr \}$. We first show that $\ZZ_B \cap \{pq, pr, qr\} =\emptyset$.


\begin{lem}\label{lem:empty set}
	$\ZZ_B \cap \{pq, pr, qr\} =\emptyset$.
\end{lem}
\begin{proof}
	Observe that $\sharp(\ZZ_B \cap \{pq, pr, qr\})$ takes the value $0$, $1$, $2$ or $3$. If $\sharp(\ZZ_B \cap \{pq, pr, qr\})=3$, then $\sharp B=pqr$. This means $A=B=\Zpqr$, which contradicts $\ZZ_A \cap \{pq, pr, qr\} =\emptyset$. If $\sharp (\ZZ_B \cap \{pq, pr, qr\}) =2$, then due to Remark \ref{Rem:2 is 2}, we have $\sharp (\ZZ_A \cap \{pq, pr, qr\}) =2$ which is a contradiction. It suffices to show $\sharp (\ZZ_B \cap \{pq, pr, qr\})\not=1$. We will prove it by contradiction. Assume that $\sharp (\ZZ_B \cap \{pq, pr, qr\})=1$. Without loss of generality, we suppose that $pq\in \ZZ_B$ and $pr,qr\notin \ZZ_B$. If follows that $r$ divides $\sharp B$, implying $\sharp B \ge r$. We thus consider two cases: $\sharp B>r$ and $\sharp B=r$, and prove that such a set $B$ does not exist in each of these two cases.
	
	If $\sharp B>r$, then by pigeonhole principle,  there exist two different elements $b,b'\in B$ such that $r~|~b-b'$. It follows that 
	$$
	(B-B)\cap  (r\Zpqr^\star\cup pr\Zpqr^\star \cup qr\Zpqr^\star ) \not=\emptyset.
	$$
	Since $pr,qr\notin \ZZ_A$ and $(A,B)$ is a spectral pair, we have that $$(B-B)\cap r\Zpqr^\star\not=\emptyset$$ and consequently $r\in \ZZ_A$. It follows from Lemma \ref{lem:OLALALA} that $r>p$, $r>q$ and $p,q\in \ZZ_B$. Since $p\in \ZZ_B$, $pr\notin \ZZ_B$ and $pq\notin \ZZ_A$, by Lemma \ref{lem:OLALALA}, we have $p>r$ which is impossible.
	
	Now suppose that $\sharp B=r$. If $A$ has the form
	$$
	A=\{(x_j,y_j,j):0\le j\le r-1 \},
	$$ 
	for some $x_j\in \Z_p$ and some $y_j\in \Z_q$ for all $0\le j\le r-1$, then by Lemma \ref{lem:prime cycle 2}, we have that $pqA$ is a $r$-cycle which implies that $pq\in \ZZ_A$. This is impossible. Thus there exist two different elements $a,a'\in A$ such that $r~|~a-a'$. It follows that
	$$
	(A-A)\cap (r\Zpqr^\star\cup pr\Zpqr^\star \cup qr\Zpqr^\star)  \not=\emptyset.
	$$
	Since $pr, qr\notin \ZZ_B$, we have $r\in \ZZ_B$. By the fact that $pr,qr\notin \ZZ_A$, $pr,qr\notin \ZZ_B$ and Lemma \ref{lem:OLALALA}, we have $r>p$, $r>q$ and $p,q\in \ZZ_A$. Since $p\in \ZZ_A$, $pq\notin \ZZ_A$ and $pr\notin \ZZ_B$, we have $p>q$. Similarly, since $q\in \ZZ_A$, $pq\notin \ZZ_A$ and $qr\notin \ZZ_B$, we have $q>p$ which is a contradiction.
	
	We conclude that $\ZZ_B \cap \{pq, pr, qr\}$ must be empty.

\end{proof}


Now we claim that $p,q,r\notin \ZZ_A$. In fact, if $p\in \ZZ_A$, then it follows from Lemma \ref{lem:OLALALA} that $p>q$ and $q\in \ZZ_B$. By Lemma \ref{lem:OLALALA} again, we have $q>p$ which is a contradiction. This completes the proof of our claim. Similarly, we have $p,q,r\notin \ZZ_B$.

Now we prove that $\sharp A=1$. If $\sharp A>1$, then $B$ has two different elements. By the fact that $\ZZ_A \cap \{p,q,r,pq,pr,qr \}=\emptyset$ and that $(B-B)\setminus\{0\} \subset \ZZ_A$, we have $1\in \ZZ_A$. Similarly, we have $1\in \ZZ_B$. This implies that for any two different elements $(x,y,z), (x',y',z') \in A$, we must have $x\not=x', y\not=y'$ and $z\not=z'$. It follows that $\sharp A\le \min\{p, q, r\} $. Without loss of generality, we assume $p<q<r$. Then have $\sharp A\le p$. If $\sharp A=p$, then $qrA$ has the from
$$
\{(j,0,0): 0\le j\le p-1 \}.
$$
It is easy to see that $qrA$ is a $p$-cycle and consequently $qr\in \ZZ_A$. This is impossible. Thus $\sharp A<p$. However, by Theorem \ref{Lem:Lam and Leung}, $\sharp A$ is a nonnegative integer combination of $p, q$ and $r$, that is, $\sharp A\ge \min\{p, q, r\}=p$. This is a contradiction. Thus we have $\sharp A=1$.  Obviously, the set $A$ satisfies the conditions $\Ta$ and $\Tb$.


%

\subsection{Case 3: $\sharp (\ZZ_A \cap \{pq, pr, qr\}) =1$}\label{subsec: 1}
Obviously, the condition $\Tb$ holds for $A$ vacuously. It remains to prove that $A$ satisfies the condition $\Ta$. Without loss of generality, we suppose that $qr\in \ZZ_A$ and $pq, pr\notin \ZZ_A$. Then the prime number $p$ divides $\sharp A$. 
We claim that $pq,pr\notin \ZZ_B$. In fact, if $pr\in \ZZ_B$, then $\sharp B$ is divided by $q$ and consequently $\sharp B\ge pq$. Since $pq\notin \ZZ_A$, by Lemma \ref{lem:1}, $B$ has the form (\ref{eq:f(pq)}) with $S=\Z_p\times \Z_q$. By Lemma \ref{lem:3}, we have $\sharp (\ZZ_B \cap \{pq, pr, qr\}) =2$. But by Remark \ref{Rem:2 is 2}, we have $\sharp (\ZZ_A \cap \{pq, pr, qr\}) =2$ which is a contradiction. Thus $pr\notin \ZZ_B$. Similarly, we have $pq\notin \ZZ_B$. On the other hand, by Lemma \ref{lem:empty set} that $\ZZ_B \cap \{ pq, pr, qr\}=\emptyset$ implying $\ZZ_A \cap \{ pq, pr, qr\}=\emptyset$, we have 
$\ZZ_B \cap \{pq, pr, qr\} \not=\emptyset$. Since $pq,pr\notin \ZZ_B$, we obtain
$qr\in \ZZ_B$.

If $\sharp A=p$, then both $A$ and $B$ satisfy $\Ta$. Assume that $\sharp A>p$. By the pigeonhole principle, there exist two different elements $b,b'\in B$ such that $p~|~b-b'$. It follows that 
$$
\ZZ_A\cap \{p,pq,pr \} \not=\emptyset.
$$
Since $pq,pr\notin \ZZ_A$, we have $p\in \ZZ_A$.
Applying Lemma \ref{lem:OLALALA} by the fact that $p\in \ZZ_A$, $pq\notin \ZZ_A$ and $pr\notin \ZZ_B$,  we have $p>q$ and $r\in \ZZ_B$. Similarly, we have $p>r$ and $q\in \ZZ_B$. Moreover, since $r\in \ZZ_B$ and $pr\notin \ZZ_B$, we obtain that $rB$ has a $p$-cycle. By Lemma \ref{lem:lalala}, we have $r>p$, which is a contradiction. Thus, we have $\sharp A=p$ which completes the proof. 

\section*{Acknowledgments} 
We would like to thank Romanos-Diogenes Malikiosis for bringing our attention to Tao's blog \cite{t2}. We are also grateful to the anonymous reviewer's valuable remarks.

\bibliographystyle{amsplain}


\begin{dajauthors}
\begin{authorinfo}[ruxi]
  Ruxi Shi\\
  Department of Mathematical sciences, University of Oulu\\
  Oulu, Finland\\
  ruxi\imagedot{}shi\imageat{}oulu\imagedot{}fi\\  
  {\em Current address:}\\
  Institute of Mathematics, Polish Academy of Sciences\\
  ul. \'Sniadeckich
  8,00-656 Warszawa,Poland\\
  rshi\imageat{}impan\imagedot{}pl \\
  \url{https://sites.google.com/view/ruxishi}
\end{authorinfo}
\end{dajauthors}

\end{document}